\documentclass[oribibl]{llncs}
\usepackage{booktabs}
\usepackage{tikz}
\usepackage{amssymb,amsmath}
\usepackage{url}

\newcommand{\legendre}[2]{\genfrac {(}{)}{1pt}{}{#1}{#2}}
\newcommand {\F}{\mathbb {F}}
\newcommand{\Fp}{\F_p}
\newcommand{\Fq}{\F_q}
\newcommand{\C}{\mathbb{C}}

\newcommand{\Z}{\mathbb{Z}}

\newcommand{\M}{\textsf{M}}
\newcommand{\Cl}{\operatorname{Cl}}
\newcommand{\Ell}{\operatorname{Ell}}
\newcommand{\End}{\operatorname{End}}
\newcommand{\Gal}{\operatorname{Gal}}
\newcommand{\Res}{\operatorname{Res}}
\newcommand{\f}{\mathfrak{f}}

\newcommand{\w}{\mathfrak{w}}
\newcommand{\af}{\mathfrak{a}}
\newcommand{\bfrak}{\mathfrak{b}}
\newcommand{\lf}{\mathfrak{l}}
\newcommand{\pf}{\mathfrak{p}}
\newcommand{\nf}{\mathfrak{n}}

\newcommand{\WN}{W|_N}
\newcommand{\EllD}{\mathrm {Ell}_{\O}}
\renewcommand{\O}{\mathcal{O}}
\newcommand{\rmPhi}{{\rm\Phi}}
\newcommand{\rmPsi}{{\rm\Psi}}
\newcommand{\rmGamma}{{\rm\Gamma}}

\newtheorem{algorithm}{Algorithm}
\newcommand{\algstart}[2]{\smallskip\noindent{\bf Algorithm} #1. \emph{#2}\begin{enumerate}}
\newcommand{\algend}{\end{enumerate}\vspace{4pt}}

\begin{document}

\title{Class invariants by the CRT method}
\author{Andreas Enge$^1$\and Andrew V. Sutherland$^2$}
\institute{$^1$INRIA Bordeaux--Sud-Ouest\hspace{20pt}$^2$Massachusetts Institute of Technology}

\maketitle

\begin{abstract}
We adapt the CRT approach for computing Hilbert class polynomials to handle
a wide range of class invariants.
For suitable discriminants $D$, this improves its performance by a large constant
factor, more than 200 in the most favourable circumstances.
This has enabled record-breaking constructions of elliptic curves
via the CM method, including examples with $|D|>10^{15}$.
\end{abstract}

\section{Introduction}
Every ordinary elliptic curve $E$ over a finite field $\Fq$
has \textit {complex multiplication} by an imaginary quadratic order $\O$,
by which we mean that the endomorphism ring $\End(E)$
is isomorphic to $\O$. The Deuring lifting theorem implies that $E$ is the
reduction of an elliptic curve $\hat{E}/\C$ that also has complex multiplication by~$\O$.
Let $K$ denote the fraction field of $\O$.
The $j$-invariant of $\hat{E}$ is an algebraic integer whose minimal polynomial
over $K$ is the \textit {Hilbert class polynomial} $H_D$, where
$D$ is the discriminant of $\O$.  Notably, the polynomial $H_D$ actually lies in $\Z[X]$,
and its splitting field is the \textit {ring class field} $K_\O$ for the order $\O$.

Conversely, an elliptic curve $E/\Fq$ with complex multiplication by $\O$
exists whenever $q$ satisfies the norm equation $4q=t^2-v^2D$, with $t,v\in\Z$
and $t\not\equiv 0$ modulo the characteristic of $\Fq$.
In this case $H_D$ splits completely over $\Fq$, and its roots are precisely the
$j$-invariants of the elliptic curves $E/\Fq$ that have complex multiplication
by~$\O$.
Such a curve has $q+1\pm t$ points, where $t$ is determined, up to a
sign, by the norm equation.  With a judicious selection of $D$ and $q$
one may obtain a curve with prescribed order.
This is known as the \emph{CM~method}.

The main challenge for the CM method is to obtain the polynomial $H_D$, which
has degree equal to the class number $h(D)$, and total size
$O(|D|^{1+\epsilon})$.
There are three approaches to computing $H_D$, all of which, under
reasonable assumptions, can achieve a running time of $O(|D|^{1+\epsilon})$.
These include the complex analytic method \cite{Enge:FloatingPoint}, a $p$-adic
algorithm \cite{Couveignes:ClassPolynomial,Broker:pAdicClassPolynomial},
and an approach based on the Chinese Remainder Theorem (CRT)
\cite{Belding:HilbertClassPolynomial}.
The first is the most widely used, and it is quite efficient; the
range of discriminants to which it may be applied is limited not
by its running time, but by the space required.
The polynomial $H_D$ is already likely to exceed available memory when
$|D|>10^9$, hence one seeks to apply the CM method to alternative class polynomials that
have smaller coefficients than~$H_D$.  This makes computations with $|D|>10^{10}$
feasible.

Recently, a modified version of the CRT approach was proposed that greatly
reduces the space required for the CM method
\cite{Sutherland:HilbertClassPolynomials}.
Under the Generalised Riemann Hypothesis (GRH), this algorithm is able to
compute $H_D\bmod P$ using $O(|D|^{1/2+\epsilon}\log P)$ space and
$O(|D|^{1+\epsilon})$ time.
(Here and in the following, all complexity estimates refer to bit operations.)
The reduced space complexity allows it to handle much
larger discriminants, including examples with $|D| > 10^{13}$.

An apparent limitation of the CRT approach is that it depends on some specific
features of the $j$-function. As noted in \cite{Belding:HilbertClassPolynomial},
this potentially precludes it from computing class polynomials other than $H_D$.
The purpose of the present article is to show how these obstructions may be
overcome, allowing us to apply the CRT method to many functions other than $j$,
including two infinite families.

Subject to suitable constraints on $D$, we may then
compute a class polynomial with smaller coefficients
than $H_D$ (by a factor of up to 72), and, in certain cases, with
smaller degree (by a factor of 2).
Remarkably, the actual running time with the CRT method is
typically \textit {better} than the size difference would suggest.
Fewer CRT moduli are needed, and we may choose a subset for
which the computation is substantially faster than on average.

We start~\S\ref {sec:crtj} with a brief overview of the CRT method, and then
describe a new technique to improve its performance, which also turns out
to be crucial for certain class invariants.
After discussing families of invariants in~\S\ref {sec:invariants},
we consider CRT-based approaches applicable to the different families and
give a general algorithm in~\S\ref {sec:crtinv}.  Computational results and performance
data appear in~\S\ref {sec:implementation}.

\section{Hilbert class polynomials via the CRT}
\label {sec:crtj}

\subsection {The algorithm of Belding, Br\"oker, Enge, Lauter and Sutherland}
\label {ssec:bbels}

The basic idea of the CRT-based algorithm for Hilbert class
polynomials is to compute $H_D$ modulo many small primes~$p$,
and then lift its coefficients by Chinese remaindering to integers,
or to their reductions modulo a large (typically prime) integer~$P$,
via the explicit CRT \cite[Thm. 3.1]{Bernstein:ModularExponentiation}.  The latter
approach suffices for most applications, and while it does not substantially
reduce the running time (the same number of small primes is required), it
can be accomplished using only $O(|D|^{1/2+\epsilon}\log P)$ space with
the method of \cite[\S6]{Sutherland:HilbertClassPolynomials}.

For future reference, we summarise the algorithm to compute $H_D\bmod p$
for a prime $p$ that splits completely in the ring class field $K_\O$.  Let $h=h(D)$.

\begin{algorithm}[Computing $\boldsymbol{H_D\bmod p}$]
\label {alg:crtj}
\quad
\begin{enumerate}
\item
Find the $j$-invariant $j_1$ of an elliptic curve $E/\Fp$ with $\End(E)\cong\O$.
\item
Enumerate the other roots $j_2, \ldots, j_{h}$ of $H_D \bmod p$.
\item
Compute $H_D(X) \bmod p = (X-j_1) \cdots (X-j_{h})$.
\end{enumerate}
\end {algorithm}

The first step is achieved by varying $j_1$ (systematically or
randomly) over the elements of $\Fp$ until it corresponds to a suitable
curve; details and many practical improvements are given in
\cite{Belding:HilbertClassPolynomial,Sutherland:HilbertClassPolynomials}.
The third step is a standard building block of computer algebra.
Our interest lies in Step~2.

\subsection{Enumerating the roots of $H_D\bmod p$}
\label{ssec:crtjenum}

The key idea in \cite{Belding:HilbertClassPolynomial} leading to a
quasi-linear complexity is to apply the Galois action of
$\Cl(\O) \simeq \Gal (K_\O/K)$.  The group $\Cl(\O)$ acts on the roots of $H_D$,
and when $p$ splits completely in $K_\O$ there is a corresponding
action on the set $\EllD(\Fp)=\{j_1,\ldots,j_h\}$ containing the roots of
$H_D \bmod p$.  For an ideal class $[\af]$ in $\Cl(\O)$ and a $j$-invariant
$j_i\in\EllD(\Fp)$,
let us write $[\af]j_i$ for the image of $j_i$ under the Galois
action of $[\af]$.  We then have $\EllD(\Fp)=\{[\af]j_1:[\af]\in\Cl(\O)\}$.

As in \cite[\S 5]{Sutherland:HilbertClassPolynomials}, we
use a polycyclic presentation defined by a sequence
of ideals $\lf_1,\ldots,\lf_m$ with prime norms
$\ell_1,\ldots,\ell_m$ whose classes generate $\Cl(\O)$.  The \emph{relative order}
$r_k$ is the least positive integer for which
$[\lf_k^{r_k}] \in \langle [\lf_1], \ldots, [\lf_{k-1}] \rangle$.
We may then uniquely write
$[\af] = [\lf_1^{e_1}]\cdots[\lf_m^{e_m}]$, with $0 \leq e_k < r_k$.
To maximise performance, we use a presentation
in which $\ell_1<\cdots <\ell_m$, with each~$\ell_k$
as small as possible subject to $r_k > 1$.  Note that the relative order $r_k$
divides the order~$n_k$ of $[\lf_k]$ in $\Cl(\O)$, but for $k>1$ we can
(and often do) have $r_k < n_k$.

For each $j_i\in\EllD(\Fp)$ and each $\O$-ideal $\lf$ of prime norm $\ell$, the $j$-invariant~$[\lf]j_i$ corresponds to an
$\ell$-isogenous curve, which we may obtain as a root of
$\rmPhi_{\ell}(j_i,X)$, where $\rmPhi_\ell\in\Z[J,J_{\ell}]$ is the
\textit {classical modular polynomial} \cite[\S 69]{Weber:Algebra}.
The polynomial~$\rmPhi_\ell$ has the pair of functions $\bigl(j(z),j(\ell z)\bigr)$
as roots, and parameterises isogenies of degree~$\ell$.

Fixing an isomorphism $\End(E)\cong \O$, we let $\pi\in\O$ denote the
Frobenius endomorphism.  When the order~$\Z[\pi]$ is maximal at
$\ell$, the univariate polynomial
$\rmPhi_\ell(j_i,X)\in\Fp[X]$ has exactly two roots $[\lf]j_i$ and~$[\bar\lf]j_i$
when $\ell$ splits in $\O$, and a single root $[\lf]j_i$ if $\ell$ is
ramified \cite[Prop.~23]{Kohel:Thesis}.  To simplify matters, we assume here
that~$\Z[\pi]$ is maximal at each $\ell_k$, but this is not necessary,
see \cite[\S 4]{Sutherland:HilbertClassPolynomials}.

We may enumerate $\EllD(\Fp)=\{[\af]j_1:[\af]\in\langle[\lf_1],\ldots,[\lf_m]\rangle\}$
via \cite[Alg.~1.3]{Sutherland:HilbertClassPolynomials}:
\begin{algorithm}[Enumerating $\boldsymbol{\EllD(\Fp)}$ --- Step 2 of Algorithm
\ref {alg:crtj}]
\label {alg:crtjenum}
\quad
\begin{enumerate}
\item
Let $j_2$ be an arbitrary root of $\rmPhi_{\ell_m}(j_1,X)$ in $\Fp$.
\item
For $i$ from $3$ to $r_m$, let $j_i$ be the root of $\rmPhi_{\ell_m}(j_{i-1},X)/(X-j_{i-2})$ in $\Fp$.
\item
If $m > 1$, then for $i$ from $1$ to $r_m$:

\hspace{12pt} Recursively enumerate the set $\{[\af]j_i:[\af]\in\langle[\lf_1],\ldots,[\lf_{m-1}]\rangle\}$.
\end{enumerate}
\end{algorithm}
In general there are two distinct choices for $j_2$, but either will do.
Once $j_2$ is chosen, $j_3,\ldots,j_{r_m}$ are determined.  The sequence
$(j_1,\ldots,j_{r_m})$ corresponds to a path of $\ell_m$-isogenies;
we call this path an $\ell_m$-\emph{thread}.

The choice of $j_2$ in Step 1 may change the order in which $\EllD(\Fp)$ is
enumerated.  Three of the sixteen possibilities when $m=2$,
$r_1=4$, and $r_2=3$ are shown below; we assume
$[\lf_2^3]=[\lf_1]$, and label each vertex $[\lf_2^e]j_1$ by
the exponent~$e$.

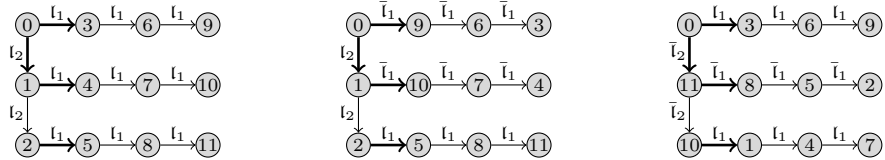
\begin{center}
\begin{tikzpicture}
\scriptsize
  \tikzstyle{vertex}=[circle,draw=black,fill=black!15,minimum size=9pt,inner sep=0pt]

  \foreach \name/\x/\y in {0/0.0/1.6, 3/0.8/1.6, 6/1.6/1.6,  9/2.4/1.6,
  												 1/0.0/0.8, 4/0.8/0.8, 7/1.6/0.8, 10/2.4/0.8,
  												 2/0.0/0.0, 5/0.8/0.0, 8/1.6/0.0, 11/2.4/0.0}
    \node[vertex] (G-\name) at (\x,\y) {$\name$};
  \foreach \from/\to in {0/1,0/3,1/4,2/5}
    \draw[line width=1.0pt,->] (G-\from) -- (G-\to);
  \foreach \y/\z in {1.23/black,0.43/black} \node at (-0.15,\y) {{\color{\z}$\lf_2$}};
  \foreach \x/\z in {0.4/black,1.2/black,2.0/black} \node at (\x,1.76) {{\color{\z}$\lf_1$}};
  \foreach \x/\z in {0.4/black,1.2/black,2.0/black} \node at (\x,0.95) {{\color{\z}$\lf_1$}};
  \foreach \x/\z in {0.4/black,1.2/black,2.0/black} \node at (\x,0.15) {{\color{\z}$\lf_1$}};
  \foreach \from/\to in {1/2,3/6,6/9,4/7,7/10,5/8,8/11}
    \draw[black,->] (G-\from) -- (G-\to);

  \foreach \name/\x/\y in {0/4.4/1.6, 9/5.2/1.6, 6/6.0/1.6,  3/6.8/1.6,
  												 1/4.4/0.8, 10/5.2/0.8, 7/6.0/0.8, 4/6.8/0.8,
  												 2/4.4/0.0, 5/5.2/0.0, 8/6.0/0.0, 11/6.8/0.0}
    \node[vertex] (G-\name) at (\x,\y) {$\name$};
  \foreach \from/\to in {0/1,0/9,1/10,2/5}
    \draw[line width=1.0pt,->] (G-\from) -- (G-\to);
  \foreach \y/\z in {1.23/black,0.43/black} \node at (4.25,\y) {{\color{\z}$\lf_2$}};
  \foreach \x/\z in {4.8/black,5.6/black,6.4/black} \node at (\x,1.79) {{\color{\z}$\bar\lf_1$}};
  \foreach \x/\z in {4.8/black,5.6/black,6.4/black} \node at (\x,0.98) {{\color{\z}$\bar\lf_1$}};
  \foreach \x/\z in {4.8/black,5.6/black,6.4/black} \node at (\x,0.15) {{\color{\z}$\lf_1$}};
  \foreach \from/\to in {1/2,9/6,6/3,10/7,7/4,5/8,8/11}
    \draw[black,->] (G-\from) -- (G-\to);

  \foreach \name/\x/\y in {0/8.8/1.6, 3/9.6/1.6, 6/10.4/1.6,  9/11.2/1.6,
  												 11/8.8/0.8, 8/9.6/0.8, 5/10.4/0.8, 2/11.2/0.8,
  												 10/8.8/0.0, 1/9.6/0.0, 4/10.4/0.0, 7/11.2/0.0}
    \node[vertex] (G-\name) at (\x,\y) {$\name$};
  \foreach \from/\to in {0/11,0/3,11/8,10/1}
    \draw[line width=1.0pt,->] (G-\from) -- (G-\to);
  \foreach \y/\z in {1.25/black,0.45/black} \node at (8.65,\y) {{\color{\z}$\bar\lf_2$}};
  \foreach \x/\z in {9.2/black,10.0/black,10.8/black} \node at (\x,1.76) {{\color{\z}$\lf_1$}};
  \foreach \x/\z in {9.2/black,10.0/black,10.8/black} \node at (\x,0.98) {{\color{\z}$\bar\lf_1$}};
  \foreach \x/\z in {9.2/black,10.0/black,10.8/black} \node at (\x,0.15) {{\color{\z}$\lf_1$}};
  \foreach \from/\to in {11/10,3/6,6/9,8/5,5/2,1/4,4/7}
    \draw[black,->] (G-\from) -- (G-\to);

\normalsize
\end{tikzpicture}
\end{center}

Bold edges indicate where a choice was made.  Regardless of these choices,
Algorithm~\ref {alg:crtjenum} correctly enumerates
$\EllD(\Fp)$ in every case \cite[Prop.~5]{Sutherland:HilbertClassPolynomials}.

\subsection {Finding roots with greatest common divisors (gcds)}
\label {ssec:gcd}

The potentially haphazard manner in which Algorithm~\ref {alg:crtjenum}
enumerates $\EllD(\Fp)$ is
not a problem when computing $H_D$, but it can complicate matters when we
wish to compute other class polynomials.  We could distinguish the actions
of $\lf$ and $\bar\lf$ using an Elkies kernel polynomial \cite{Elkies98},
as suggested in \cite[\S 5]{Broker:pAdicClassPolynomial}, however this slows
down the algorithm significantly.  An alternative approach using polynomial
gcds turns out to be much more efficient, and actually speeds up
Algorithm~\ref {alg:crtjenum}, making it already a useful improvement when
computing $H_D$.

We need not distinguish the actions of $\lf$ and $\bar\lf$ at this stage,
but we wish to ensure that our enumeration of $\EllD(\Fp)$ makes a consistent
choice of direction each time it starts an $\ell$-thread.
The first $\ell$-thread may be oriented arbitrarily, but for each subsequent
$\ell$-thread $(j_1',j_2',\ldots,j_r')$, we apply Lemma~\ref{gcdlemma} below.
This allows us to ``square the corner"  by choosing $j_2'$ as the unique
common root of $\rmPhi_\ell(X,j_1')$ and $\rmPhi_{\ell'}(X,j_2)$, where
$(j_1,\ldots,j_r)$ is a previously computed $\ell$-thread and $j_1$ is
$\ell'$-isogenous to $j_1'$.  The edge $(j_1,j_1')$ lies in an
$\ell'$-thread that has already been computed, for some $\ell'>\ell$.

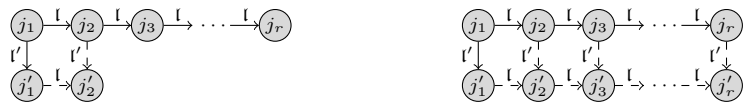
\begin{center}
\begin{tikzpicture}
\scriptsize
  \tikzstyle{vertex}=[circle,draw,fill=black!15,minimum size=12pt,inner sep=0pt]

  \foreach \name/\lab/\x/\y in {1/j_1/0.0/1.6, 2/j_2/0.8/1.6, 3/j_3/1.6/1.6,  r/j_r/3.3/1.6,
  												      11/j_1'/0.0/0.8, 22/j_2'/0.8/0.8}
    \node[vertex] (G-\name) at (\x,\y) {$\lab$};
  \node (G-4) at (2.3,1.6) {};
  \node (G-dots) at (2.5,1.6) {$\cdots$};
  \node (G-rm1) at (2.6,1.6) {};
  \foreach \from/\to in {1/11,1/2}
    \draw[->] (G-\from) -- (G-\to);
  \foreach \x in {-0.15,0.65} \node at (\x,1.22) {$\lf'$};
  \foreach \x in {0.4,1.2,2.0,2.9} \node at (\x,1.75) {$\lf$};
  \node at (0.4,0.95) {{\color{black}$\lf$}};
  \foreach \from/\to in {2/3,3/4,rm1/r}
    \draw[black,->] (G-\from) -- (G-\to);
  \foreach \from/\to in {11/22,2/22}
    \draw[dashed,->] (G-\from) -- (G-\to);

  \foreach \name/\lab/\x/\y in {1/j_1/6.0/1.6, 2/j_2/6.8/1.6, 3/j_3/7.6/1.6,  r/j_r/9.3/1.6,
  												      11/j_1'/6.0/0.8, 22/j_2'/6.8/0.8, 33/j_3'/7.6/0.8, rr/j_r'/9.3/0.8}
    \node[vertex] (G-\name) at (\x,\y) {$\lab$};
  \node (G-4) at (8.3,1.6) {};
  \node (G-dots) at (8.5,1.6) {$\cdots$};
  \node (G-rm1) at (8.6,1.6) {};
  \node (G-44) at (8.3,0.8) {};
  \node (G-ddots) at (8.5,0.8) {$\cdots$};
  \node (G-rrm1) at (8.6,0.8) {};
  \foreach \from/\to in {1/11,1/2}
    \draw[->] (G-\from) -- (G-\to);
  \foreach \x in {5.85,6.65,7.45,9.15} \node at (\x,1.22) {$\lf'$};
  \foreach \x in {6.4,7.2,8.0,8.9} \node at (\x,1.75) {$\lf$};
  \foreach \x in {6.4,7.2,8.0,8.9} \node at (\x,0.95) {$\lf$};
  \foreach \from/\to in {2/3,3/4,rm1/r}
    \draw[black,->] (G-\from) -- (G-\to);
  \foreach \from/\to in {11/22,2/22,22/33,3/33,33/ddots,rrm1/rr,r/rr}
    \draw[dashed,->] (G-\from) -- (G-\to);

\normalsize
\end{tikzpicture}
\end{center}

Having computed $j_2'$, we could compute $j_3',\ldots,j_r'$ as before,
but it is usually better to continue using gcds, as depicted above.
Asymptotically, both root-finding and gcd computations are
dominated by the $O(\ell^2 \M (\log p))$ time it takes to instantiate $\rmPhi_{\ell}(X,j_i)\bmod p$,
but in practice $\ell$ is small, and we effectively gain a factor
of $O(\log p)$ by using gcds when $\ell\approx \ell'$.
This can substantially reduce the running time of Algorithm~\ref {alg:crtjenum},
as may be seen in Table~\ref {tab1} of~\S\ref {sec:implementation}.

With the gcd approach described above, the total number of root-finding operations
can be reduced from $\prod_{k=1}^m r_k$ to $\sum_{k=1}^m r_k$.  When $m$ is large,
this is a big improvement, but it is no help when $m=1$, as necessarily occurs
when $h(D)$ is prime. However, even in this case we can apply gcds by looking
for an auxiliary ideal $\lf_1'$, with prime norm $\ell_1'$, for which
$[\lf_1']=[\lf_1^{e}]$.  When $r_1$ is large, such an $\lf_1'$ is easy to
find, and we may choose the best combination of $\ell_1'$ and $e$ available.  This
idea generalises to $\ell_k$-threads, where we seek
$[\lf_k']\in\langle[\lf_1]\ldots,[\lf_k]\rangle
\backslash \langle[\lf_1]\ldots,[\lf_{k-1}]\rangle$.

\begin{lemma}\label{gcdlemma}
Let $j_1,j_2\in\EllD(\Fp)$, and let $\ell_1,\ell_2\ne p$ be distinct primes with
$4\ell_1^2\ell_2^2 < |D|$.  Then $\gcd\bigl(\rmPhi_{\ell_1}(j_1,X),\rmPhi_{\ell_2}(j_2,X)\bigr)$
has degree at most~$1$.
\end{lemma}
\begin{proof}
It follows from \cite[Prop.~23]{Kohel:Thesis} that
$\rmPhi_{\ell_1}(X,j_1)$ and $\rmPhi_{\ell_2}(X,j_2)$ have at most two common
roots in the algebraic closure $\overline \F_p$, which in fact lie in $\EllD(\Fp)$.
If there are exactly two, then both $\ell_1 = \lf_1 \overline\lf_1$ and
$\ell_2 = \lf_2 \overline\lf_2$ split in $\O$, and one of
$\lf_1^2\lf_2^2$ or $\lf_1^2\bar\lf_2^2$ is principal with a
non-rational generator.  We thus have a norm equation
$4\ell_1^2\ell_2^2=a^2-b^2D$ with $a,b \in \Z$ and
$b\ne 0$, and the lemma follows.
\end{proof}

\section{Class invariants}
\label {sec:invariants}

Due to the large size of $H_D$, much effort has been spent seeking
smaller generators of $K_\O$.
For a modular function $f$ and $\O=\Z[\tau]$, with $\tau$ in the upper half plane,
we call $f (\tau)$ a
\textit {class invariant} if $f (\tau) \in K_\O$.
The \textit {class polynomial} for $f$ is
\[
H_D [f] (X) = \prod_{[\af] \in \Cl (\O)} \left(  X - [\af] f (\tau) \right).
\]
The contemporary tool for determining class invariants is Shimura's reciprocity
law; see \cite[Th.~4]{Schertz02} for a fairly general result.  Class invariants
arising from many different modular functions have been described in the literature;
we briefly summarise some of the most useful ones.

Let $\eta$ be Dedekind's function, and  let $\zeta_n=\exp(2\pi i/n)$.
Weber considered
\[
\f = \zeta_{48}^{-1} \frac {\eta \left( \frac {z+1}{2} \right)}{\eta (z)},
\qquad \f_1 (z) = \frac {\eta \left( \frac {z}{2} \right)}{\eta (z)},
\qquad \f_2 (z) = \sqrt 2 \, \frac {\eta (2 z)}{\eta (z)},
\]
powers of which yield class invariants when $\legendre {D}{2} \neq -1$, and also
$\gamma_2 = \sqrt[3] j$, which is a class invariant whenever $3\nmid D$.
The Weber functions can be generalised \cite {EnMo09,EnSc04,GeSt98,Gee01,HaVi97},
and we have the simple and double $\eta$-quotients
\[
\w_N (z) = \frac {\eta \left( \frac {z}{N} \right)}{\eta (z)};\qquad\qquad
\w_{p_1, p_2} = \frac {\eta \left( \frac {z}{p_1} \right) \eta \left( \frac
{z}{p_2} \right)}{\eta \left( \frac {z}{p_1 p_2} \right) \eta (z)}
\text { with } N = p_1 p_2,
\]
where $p_1$ and $p_2$ are primes.  Subject to constraints on $D$, including
that no prime dividing $N$ is inert in $\O$,
suitable powers of these functions yield class invariants, see
\cite{EnMo09,EnSc04}.
For $s = 24/\gcd \bigl(24, (p_1 -1)(p_2 - 1)\bigr)$, the canonical power
$\w_{p_1, p_2}^s$ is invariant under the Fricke involution $\WN:z \mapsto \frac {-N}{z}$
for $\rmGamma^0 (N)$, equivalently, the Atkin-Lehner involution of level $N$, by \cite[Thm.~2]{EnSc05}.

The theory of \cite {Schertz02} applies to any functions for
$\rmGamma^0 (N)$, in particular to those of prime level $N$
invariant under the Fricke involution, which yield class invariants
when $\legendre {D}{N} \neq -1$.  Atkin developed a method
to compute such functions $A_N$, which are conjectured to have a pole of minimal order
at the unique cusp \cite{Elkies98,Morain95}.  These
are used in the SEA algorithm, and can be found in
\textsc {Magma} or \textsc {Pari/GP}.

The functions above all yield algebraic integers, so $H_D [f] \in \O_K[X]$.
Except for $\w_N^e$ or when $\gcd (N, D) \neq 1$, in which cases
additional restrictions may apply, one actually has $H_D [f] \in \Z [X]$,
cf.~\cite[Cor.~3.1]{EnSc04}.  The (logarithmic) \textit{height} of
$H_D [f]=\sum a_iX^i$ is $\log\max|a_i|$, which determines
the precision needed to compute the $a_i$.  We let
$c_D(f)$ denote the ratio of the heights of $H_D[j]$ and $H_D [f]$.

With $c(f)=\lim_{|D|\to\infty}c_D(f)$, we have: $c(\gamma_2)=3$;
$c(\f)=72$ (when $\legendre {D}{2} = 1$);
\[
c(\w^e_N)=\frac{24(N+1)}{e(N-1)};\qquad
c(\w_{p_1,p_2}^s)=\frac{12\psi(p_1p_2)}{s(p_1-1)(p_2-1)};\qquad
c(A_N)=\frac{N+1}{2|v_N|},
\]
where
$e$ divides the exponent~$s$ defined above,
$v_N$ is the order of the pole of $A_N$ at the cusp,
and $\psi(p_1p_2)$ is $(p_1+1)(p_2+1)$ when $p_1\ne p_2$, and $p_1(p_1+1)$
when $p_1=p_2$.  Morain observed in \cite{Morain09} that $c(A_{71})=36$, which
is so far the best value known when $\legendre {D}{2} = -1$.  We conjecture that
in fact for all primes $N > 11$ with $N \equiv 11 \bmod{60}$ we have
$c(A_N) = 30 \frac {N+1}{N-11}$, and that for $N \equiv -1 \bmod {60}$
we have $c(A_N)=30$.
This implies that given an arbitrary discriminant $D$, we can
always choose $N$ so that $A_N$ yields class invariants with $c_D(A_N)\ge 30+o(1)$.

When the prime divisors of $N$ are all ramified in $K$, both $\w_{p_1,p_2}$ and $A_N$
yield class polynomials that are squares in $\Z [X]$, see
\cite[\S 1.6]{Enge07} and \cite{EnSc10}. Taking the square root of such
a class polynomial reduces both its degree and its height by a factor of 2.
For a composite fundamental discriminant~$D$
(the most common case), this applies to $H_D[A_N]$ for any prime $N \mid D$.  In the
best case, $D$ is divisible by 71, and we obtain a class polynomial that is
144 times smaller than $H_D$.

\subsection {Modular polynomials}

Each function $f(z)$ considered above is related to $j(z)$ by a
modular polynomial $\rmPsi_f\in\Z [F, J]$ satisfying $\rmPsi_f (f (z), j (z)) = 0$.
For primes $\ell$ not dividing the level $N$, we let $\rmPhi_{\ell, f}$
denote the minimal polynomial satisfying $\rmPhi_{\ell, f} (f (z), f (\ell z)) = 0$;
it is a factor of
$
\Res_{J_\ell} \bigl( \Res_J (\rmPhi_\ell (J, J_\ell), \rmPsi_f (F, J)),
\rmPsi_f (F_\ell, J_\ell) \bigr),
$
and as such, an element of $\Z [F, F_\ell]$. Thus $\rmPhi_{\ell, f}$ generalises the
classical modular polynomial $\rmPhi_\ell=\rmPhi_{\ell,j}$.

The polynomial $\rmPhi_{\ell,f}$ has degree $d(\ell+1)$ in $F$ and $F_\ell$,
where $d$ divides $\deg_J\rmPsi_f$, see~\cite[\S6.8]{Broker:Thesis},
and $2d$ divides $\deg_J\rmPsi_f$ when $f$ is invariant under the Fricke involution.
In general, $d$ is maximal, and $d = 1$ is achievable only in
the relatively few cases where $X_0 (N)$, respectively $X_0^+ (N)$, is of genus~$0$
and, moreover, $f$ is a hauptmodul, that is, it generates the function field of the
curve.  Happily, this includes many cases of practical interest.

The polynomial $\rmPsi_f$ characterises the analytic function
$f$ in an algebraic way; when $d=1$, the polynomials $\rmPhi_\ell$ and $\rmPhi_{\ell, f}$
algebraically characterise $\ell$-isogenies between elliptic curves given by their
$j$-invariants, or by class invariants derived from $f$, respectively.
These are key ingredients for the CRT method.

\section {CRT algorithms for class invariants}
\label {sec:crtinv}

To adapt Algorithm~\ref{alg:crtj} to class invariants arising from a modular function
$f(z)$ other than $j(z)$, we only need to consider Algorithm~\ref{alg:crtjenum}.  Our
objective is to enumerate the roots of $H_D[f]\bmod p$ for
suitable primes $p$, which we are free to choose.  This may be done in
one of two ways.  The most direct approach computes
an ``$f$-invariant" $f_1$, corresponding to $j_1$, then enumerates
$f_2,\ldots,f_h$ using the modular polynomials $\rmPhi_{\ell,f}$.
Alternatively, we may enumerate $j_1,\ldots,j_h$ as before, and from these derive
$f_1,\ldots,f_h$.  The latter approach is not as efficient, but it applies
to a wider range of functions, including two infinite families.

Several problems arise. First, an elliptic curve $E/\Fp$ with CM by~$\O$
unambiguously defines a $j$-invariant~$j_1=j(E)$, but not the corresponding $f_1$.
The $f_1$ we seek is a root of $\psi_f(X) = \rmPsi_f (X, j_1) \bmod p$,
but $\psi_f$ may have other roots, which may or may not be class invariants.
The same problem occurs for the $p$-adic lifting algorithm and can be solved
generically \cite[\S 6]{Broker:Thesis}; we describe some more efficient solutions,
which are in part specific to certain types of functions.

When $\psi_f$ has multiple roots that are class invariants, these may
be roots of distinct class polynomials.
We are generally happy to compute any one of these, but it is
imperative that we compute the reduction of ``the same" class polynomial
$H_D[f]$ modulo each prime $p$.

The lemma below helps to address these issues for at least two infinite
families of functions: the double $\eta$-quotients $\w_{p_1,p_2}$
and the Atkin functions $A_N$.

\begin{lemma}
\label{lemma:psiroots}
Let $f$ be a modular function for $\rmGamma^0 (N)$, invariant under the Fricke
involution $\WN$, such that $f (z)$ and $f \left( \frac {-1}{z} \right)$ have
rational $q$-expansions.  Let the imaginary quadratic order $\O$ have conductor
coprime to $N$ and contain an ideal $\nf = \bigl( N, \frac {B_0 + \sqrt D}{2} \bigr)$.
Let $A_0 = \frac {B_0^2 - D}{4 N}$ and $\tau_0 = \frac {-B_0 + \sqrt D}{2 A_0}$,
and assume that $\gcd (A_0, N) = 1$.
Then $f (\tau_0)$ is a class invariant, and if $f (\tau)$ is any of its conjugates
under the action of $\Gal(K_\O/K)$
we have
\[
\rmPsi_f\bigl(f(\tau),j(\tau)\bigr)=0\qquad{\text{and}}\qquad
\rmPsi_f\bigl(f(\tau),[\nf]j(\tau)\bigr)=0.
\]
\end{lemma}
\begin{proof}
By definition, $\rmPsi_f \bigl(f (z), j (z)\bigr) = 0$.  Applying
the Fricke involution yields
$
0 =
\rmPsi_f \left( (\WN f) (z), (\WN j) (z) \right) =
\rmPsi_f \left( f (z), j \left( \frac {-N}{z} \right) \right) =
\rmPsi_f \left( f (z), j \left( \frac {z}{N} \right) \right).
$
The value $f (\tau_0)$ is a class invariant by \cite[Th.~4]{Schertz02}.
By the same result, we may assume that $\tau$ is the basis quotient of an ideal
$\af = \bigl( A, \frac {-B + \sqrt D}{2} \bigr)$ with
$\gcd (A, N) = 1$ and
$B \equiv B_0 \bmod {2 N}$.  Then $\frac {\tau}{N}$ is the basis quotient of
$\af \overline\nf = \bigl( AN, \frac {-B + \sqrt D}{2} \bigr)$.
It follows that $[\nf] j (\tau) = j \left( \frac {\tau}{N} \right)$, and replacing $z$ above by $\tau$
completes the proof.
\end{proof}

If we arrange the roots of $H_D$ into a graph of $\nf$-isogeny cycles corresponding
to the action of $\nf$, the lemma yields a dual graph defined on the
roots of $H_D[f]$, in which vertices $f(\tau)$ correspond to edges
$\bigl(j(\tau),[\nf]j(\tau)\bigr)$.

In computational terms, $f (\tau)$ is a root
of $\gcd \big( \rmPsi_f\bigl(X,j(\tau)\bigr), \rmPsi_f\bigl(X,[\nf]j(\tau)\bigr) \big)$.
Generically, we expect this gcd to have no other roots modulo primes $p$ that
split completely in $K_\O$.
For a finite number of such primes, there may be additional roots.
We have observed this for $p$ dividing the conductor of the order generated by
$f (\tau)$ in the maximal order of $K_\O$.
Such primes may either be excluded from our CRT computations, or addressed by one of the techniques
described in \S\ref{ssec:indirect}.

\subsection {Direct enumeration}
\label {ssec:direct}

When the polynomials $\rmPhi_{\ell,f}$ have degree $\ell+1$
we can apply Algorithm~\ref{alg:crtjenum} with essentially no modification;
the only new consideration is that $\ell$ must not divide the level $N$,
but we can exclude such $\ell$ when choosing a polycyclic presentation for
$\Cl(\O)$.  When the degree is greater than $\ell+1$ the situation is more
complex, moreover the most efficient algorithms for computing modular
polynomials do not apply \cite{BrLaSu09,Enge09mod}, making it difficult
to obtain $\rmPhi_{\ell, f}$ unless $\ell$ is very small.  Thus in
practice we do not use $\rmPhi_{\ell,f}$ in this case; instead we apply the
methods of \S\ref{ssec:indirect} or \S\ref{ssec:general}.
For the remainder of this subsection and the next we assume that we do have
polynomials $\rmPhi_{\ell,f}$ of degree $\ell+1$ with which to enumerate
$f_1,\ldots,f_h$, and consider how to determine a starting point $f_1$, given
the $j$-invariant $j_1=j(E)$ of an elliptic curve $E/\Fp$ with CM by $\O$.

When $\psi_f(X) = \rmPsi_f (X, j_1) \bmod p$ has only one root,
our choice of $f_1$ is immediately determined.  This is usually not
the case, but we may be able to ensure it by restricting our choice
of~$p$. As an example, for $f = \gamma_2$ with $3\nmid D$, if we
require that $p \equiv 2 \bmod 3$, then $f_1$ is the unique
cube root of~$j_1$ in $\Fp$.  If we additionally have $D\equiv 1\bmod 8$
and $p\equiv 3\bmod 4$, then the equation $\gamma_2 = (\f^{24}-16)/\f^8$
uniquely determines the square of the Weber $\f$ function, by \cite[Lem.~7.3]{BrLaSu09}.
To treat $\f$ itself we need an additional trick described in \S\ref{ssec:tracetrick}.

The next simplest case occurs when only one of the roots of $\psi_f$ is
a class invariant.  This necessarily happens when $f$ is invariant under
the Fricke involution and all the primes dividing $N$ are ramified in~$\O$.
In the context of Lemma~\ref{lemma:psiroots}, each root of $H_D[f]$ then
corresponds to an isolated edge $\big(j(\tau),[\nf]j(\tau)\bigr)$
in the $\nf$-isogeny graph on the roots of $H_D$, and we
compute $f_1$ as the unique root of $\gcd\bigl(\rmPsi_f(X,j_1),\rmPsi_f(X,[\nf]j_1)\bigr)$.
In this situation $\nf=\bar\nf$, and each $f(\tau)$ occurs twice as a root
of $H_D[f]$.  By using a polycyclic presentation for $\Cl(\O)/\langle[\nf]\rangle$
rather than $\Cl(\O)$, we enumerate each double root of $H_D[f]\bmod p$ just
once.

Even when $\psi_f$ has multiple roots that are class invariants, it may
happen that they are all roots of the \emph{same} class polynomial.
This applies to the Atkin functions $f=A_N$.  When $N$ is a split prime,
there are two $N$-isogenous pairs $(j_1,[\nf]j_1)$ and $([\bar\nf]j_1,j_1)$
in $\EllD(\Fp)$, and under Lemma~\ref{lemma:psiroots} these correspond to
roots $f_1$ and $[\bar\nf]f_1$ of $\psi_f$.  Both are roots of $H_D[f]$,
and we may choose either.

The situation is slightly more complicated for the double $\eta$-quotients
$\w_{p_1,p_2}$, with $N=p_1p_2$ composite.  If $p_1=\pf_1\bar\pf_1$
and $p_2=\pf_2\bar\pf_2$ both split and $p_1\ne p_2$, then there
are four distinct $N$-isogenies corresponding to four roots of $\psi_f$.
Two of these roots are related by the action of $[\nf]=[\pf_1\pf_2]$;
they belong to the same class polynomial, which we choose as
$H_D[f]\bmod p$. The other
two are related by $[\pf_1\bar\pf_2]$ and are roots of a
different class polynomial.
We make an arbitrary choice for $f_1$, explicitly compute $[\nf] f_1$, and
then check whether it occurs among the other three roots; if not, we correct the
initial choice. The techniques of \S\ref {ssec:indirect} may be used
to efficiently determine the action of $[\nf]$.

Listed below are some of the modular functions $f$ for which the roots of
$H_D[f]\bmod p$ may be directly enumerated, with sufficient constraints
on $D$ and~$p$.  In each case $p$ splits completely in $K_\O$ and $D<-4N^2$
has conductor $u$.
\renewcommand{\labelenumi}{(\arabic{enumi})}
\begin {enumerate}
\setlength{\itemsep}{4pt}
\item
$\gamma_2$, with $3\nmid D$ and $p\equiv 2\bmod 3$;
\item
$\f^2$, with $D\equiv1\bmod 8$, $3\nmid D$, and $p\equiv 11\bmod 12$;
\item
$\w_N^s$, for $N\in\{3,5,7,13\}$ and $s=24/\gcd(24,N-1)$, with $N \mid D$
and $N\nmid u$;
\item
$\w_5^2$, with $3\nmid D$, $5 \mid D$, and $5\nmid u$;
\item
$A_N$, for $N\in\{3,5,7,11,13,17,19,23,29,31,41,47,59,71\}$, with $\legendre{D}{N}\ne -1$ and $N\nmid u$.
\item
$\w_{p_1,p_2}^s$, for $(p_1,p_2)\in\{(2,3),(2,5),(2,7),(2,13),(3,5),(3,7),(3,13),(5,7)\}$ and
$s=24/\gcd\bigl(24,(p_1-1)(p_2-1)\bigr)$, with $\legendre{D}{p_1},\legendre{D}{p_2}\ne -1$ and
$p_1,p_2\nmid u$.
\item
$\w_{3,3}^6$ with $\legendre{D}{3}=1$ and $3\nmid u$.
\end {enumerate}

\subsection {The trace trick}
\label {ssec:tracetrick}

In \S\ref{ssec:direct} we were able to treat the square of the Weber $\f$
function but not $\f$ itself.  To remedy this, we generalise a method suggested
to us by Reinier Br\"{o}ker.

We consider the situation where there are two modular functions $f$ and $f'$ that
are roots of $\rmPsi_f(X,j(z))$, both of which yield class invariants for $\O$,
and we wish to apply the direct enumeration approach.
We assume that $p$ is chosen so that $\psi_f(X)=\rmPsi_f(X,j_1)\bmod p$ has
exactly two roots, and depending on which root we take as $f_1$, we may compute the
reduction of either $H_D[f](X)$ or $H_D[f'](X)$ modulo $p$.  In the case
of Weber $\f$, we have $f'=-f$, and $H_D[f']$ differs
from $H_D[f]$ only in the sign of every other coefficient.

Consider a fixed coefficient $a_i$ of $H_D[f](X)=\sum a_iX^i$;
most of the time, the trace $t = -a_{h-1} = f_1 + \cdots + f_h$ will do
(if $f'=-f$, we need to use $a_i$ with $i\not\equiv h\bmod 2$).
The two roots $f_1$ and $f_1'$ lead to two possibilities $t$ and $t'$
modulo~$p$.  However, the elementary symmetric functions $T_1 = t + t'$ and $T_2
= t t'$ are unambiguous modulo~$p$. Computing these modulo many primes $p$
yields $T_1$ and~$T_2$ as integers (via the CRT), from which $t$ and $t'$ are obtained
as roots of the quadratic equation $X^2 - T_1 X + T_2$. If these are different,
we arbitrarily pick one of them, which, going back,
determines the set of conjugates $\{ f_1, \ldots, f_h \}$ or $\{ f_1',
\ldots, f_h' \}$ to take modulo each of the primes $p\nmid t-t'$.  In the unlikely event
that they are the same (the suspicion $t = t'$ being confirmed after, say,
looking at the second prime), we need to switch to a different coefficient $a_i$.

If $f$ and $f'$ differ by a simple transformation (such as $f'=-f$), the
second set of conjugates and the value $t'$ are obtained essentially for free.
As a special case, when $h$ is odd and the class invariants are units (as with
Weber $\f$), we can simply fix $t = a_0=1$, and need not compute $T_1=0$ and $T_2=-1$.

The key point is that the number of primes $p$ we use to determine $t$
is much less than the number of primes we use to compute $H_D[f]$.
Asymptotically, the logarithmic height of the trace is smaller than the height
bound we use for $H_D[f]$ by a factor quasi-linear in $\log |D|$, under
the GRH.  In practical terms, determining~$t$ typically requires less than
one tenth of the primes used to compute $H_D[f]$, and these computations
can be combined.

The approach described above generalises immediately to more than two roots, but this case
does not occur for the functions we examine.  Unfortunately it can be used only
in conjunction with the direct enumeration approach of \S\ref{ssec:direct};
otherwise we would have to consistently distinguish not only
between $f_1$ and $f_1'$, but also between $f_i$ and $f_i'$
for $i = 2, \ldots, h$.

\subsection {Enumeration via the Fricke involution}
\label {ssec:indirect}
For functions $f$ to which Lemma~\ref{lemma:psiroots} applies, we
can readily obtain the roots of $H_D[f]\bmod p$ without using the polynomials
$\rmPhi_{\ell,f}$.  We instead enumerate the roots of $H_D\bmod p$
(using the polynomials $\rmPhi_\ell$), and arrange them into a graph~$G$
of $\nf$-isogeny cycles, where $\nf$ is the ideal of norm $N$ appearing
in Lemma~\ref{lemma:psiroots}.  We then obtain roots of
$H_D[f]\bmod p$ by computing $\gcd\bigl(\rmPsi_f(X,j_i),\rmPsi_f(X,[\nf]j_i)\bigr)$
for each edge $(j_i,[\nf]j_i)$ in $G$.

The graph $G$ is composed of $h/n$ cycles of length $n$, where $n$ is the
order of $[\nf]$ in $\Cl(\O)$.  We assume that the $\O$-ideals of norm $N$ are
all non-principal and inequivalent (by requiring $|D|>4N^2$ if needed).
When every prime dividing $N$ is ramified in
$\O$ we have $n=2$; as noted in \S\ref{ssec:direct}, every root of $H_D[f]$
then occurs with multiplicity~$2$, and we may compute the square-root of $H_D[f]$
by taking each root just once.  Otherwise we have $n>2$.

Let $[\lf_1],\ldots,[\lf_m]$ be a polycyclic presentation for $\Cl(\O)$
with relative orders $r_1,\ldots,r_m$, as in \S\ref{ssec:crtjenum}.  For $k$ from 1 to $m$
let us fix $\lf_k=\bigl(\ell_k,\frac{-B_k+\sqrt{D}}{2}\bigr)$ with $B_k \ge 0$.
To each vector $\vec{e}=(e_1,\ldots,e_m)$ with $0\le e_k < r_k$,
we associate a unique root $j_{\vec{e}}$ enumerated by
Algorithm~\ref{alg:crtjenum}, corresponding to the path taken from $j_1$ to~$j_{\vec{e}}$,
where $e_k$ counts steps taken along an $\ell_k$-thread.  For $\vec{o}=(0,\ldots,0)$
we have $j_{\vec{o}}=j_1$, and in general
\[
j_{\vec{e}}=[\lf_1^{\sigma_1e_1}\cdots\lf_m^{\sigma_m e_m}]j_{\vec{o}},
\]
with $\sigma_k=\pm 1$.  Using the method of \S\ref{ssec:gcd} to
consistently orient the $\ell_k$-threads ensures that each $\sigma_k$
depends only on the orientation of the first $\ell_k$-thread.

To compute the graph $G$ we must determine the signs $\sigma_k$.
For those $[\lf_k]$ of order 2, we let $\sigma_k=1$.  We
additionally fix $\sigma_k=1$ for the least $k=k_0$ (if any) for which $[\lf_k]$
has order greater than 2, since we need not distinguish the actions of
$\nf$ and $\bar{\nf}$.
It suffices to show how to determine $\sigma_k$, given that we know
$\sigma_1,\ldots,\sigma_{k-1}$.  We may assume $[\lf_{k_0}]$ and
$[\lf_k]$ both have order greater than 2, with $k_0 < k \le m$.

Let $\lf$ be an auxiliary ideal of prime norm $\ell$ such that
$[\lf]=[\af\bfrak]=[\lf_1^{e_1}\cdots\lf_k^{e_k}]$, with $0 \le e_i < r_i$,
where $\bfrak=\lf_k^{e_k}$, and $[\af]$ and $[\bfrak]$ have order greater than 2.
Our assumptions guarantee that such an $\lf$ exists, by the \v{C}ebotarev density
theorem, and under the GRH, $\ell$ is relatively small \cite{Bach:ERHbounds}.
The fact that $[\af]$ and $[\bfrak]$ have order greater than 2 ensures that
$[\af\bar\bfrak]$ is distinct from $[\lf]$ and its inverse.
It follows that $\sigma_k=1$ if and only if $\rmPhi_\ell(j_{\vec{o}},j_{\vec{e}})=0$,
where $\vec{e}=(e_1,\ldots,e_k,0,\ldots,0)$.

Having determined the $\sigma_k$, we compute the unique vector
$\vec{v}=(v_1,\ldots,v_m)$ for which $[\nf]=[\lf_1^{\sigma_1v_1}\cdots\lf_m^{\sigma_mv_m}]$.
We then have $[\nf]j_{\vec{o}}=j_{\vec{v}}$, yielding the edge $(j_{\vec{o}},j_{\vec{v}})$ of $G$.
In general, we obtain the vector corresponding to $[\nf]j_{\vec{e}}$
by computing $\vec{e}+\vec{v}$ and using relations $[\lf_k^{r_k}]=[\lf_1^{x_1}\cdots\lf_{k-1}^{x_{k-1}}]$
to reduce the result, cf. \cite[\S5]{Sutherland:HilbertClassPolynomials}.

This method may be used with any function $f$ satisfying Lemma~\ref{lemma:psiroots},
and in particular it applies to two infinite families of functions:
\renewcommand{\labelenumi}{(\arabic{enumi})}
\begin {enumerate}
\setlength{\itemsep}{4pt}
\item[(8)]
$A_N$, for $N>2$ prime, with $\legendre{D}{N}\ne -1$ and $N\nmid u$.
\item[(9)]
$\w_{p_1,p_2}^s$, for $p_1,p_2$ primes not both 2, with $\legendre{D}{p_1},\legendre{D}{p_2}\ne -1$ and $p_1,p_2\nmid u$.
\end{enumerate}
As above, $u$ denotes the conductor of $D<-4N^2$.

As noted earlier, for certain primes $p$ we may have difficulty
computing the edges of $G$ when $\gcd\bigl(\rmPsi_f(X,j_i),\rmPsi_f(X,[\nf]j_i)\bigr)$
has more than one root in $\Fp$.  While we need not use such primes, it is often
easy to determine the correct root.  Here we give two heuristic techniques for doing so.

The first applies when $N$ is prime, as with the Atkin functions.
In this case problems can arise when $H_D[f]$ has repeated roots modulo~$p$.  By Kummer's
criterion, this can happen only when $p$ divides the discriminant of $H_D[f]$, and even
then, a repeated root $x_1$ is only actually a problem when it corresponds to
two alternating edges in $G$, say $(j_1,j_2)$ and $(j_3,j_4)$, with the edge $(j_2,j_3)$ between them.
In this scenario we will get two roots $x_1$ and $x_2$ of  $\gcd\bigl(\rmPsi_f(X,j_2),\rmPsi_f(X,j_3)\bigr)$.
But if we already know that $x_1$ corresponds to $(j_1,j_2)$, we can unambiguously choose $x_2$.
In each of the $N$-isogeny cycles of~$G$, it is enough to find a single edge that yields
a unique root.  If no such edge exists, then every edge must yield the \emph{same} two
roots $x_1$ and~$x_2$, and we count each with multiplicity $n/2$.

The second technique applies when the roots of $H_D[f]$ are units, as with the double
$\eta$-quotients \cite[Thm.~3.3]{EnSc04}.  The product of the roots is then $\pm 1$.
Assuming that the number of edges in $G$ for which multiple roots arise is small (it is
usually zero, and rarely more than one or two), we simply test all the possible choices
of roots and see which yield $\pm 1$.  If only one combination works, then the correct
choices are determined. This is not guaranteed to happen, but in practice it almost
always does.

\subsection {A general algorithm}
\label{ssec:general}
We now briefly consider the case of an arbitrary modular function $f$
of level $N$, and sketch a general algorithm to compute $H_D[f]$ with the
CRT method.

Let us assume that $f(\tau)$ is a class invariant, and let $D$ be the
discriminant and $u$ the conductor of the order $\O=[1,\tau]$.  The
roots of $\rmPsi_f(X,j(\tau))\in K_\O[X]$ lie in the ray class field
of conductor $u N$ over $K$,
and some number $n$ of these, including $f(\tau)$, actually lie in the
ring class field $K_\O$.  We may determine~$n$ using the
method described in \cite[\S 6.4]{Broker:Thesis}, which computes the action of $(\O/N\O)^*/\O^*$
on the roots of $\rmPsi_f(X,j(\tau))$.  We note that the complexity of this
task is essentially fixed as a function of $|D|$.

Having determined $n$, we use Algorithm~\ref{alg:crtjenum} to enumerate
the roots $j_1,\ldots,j_h$ of $H_D\bmod p$ as usual, but if for any $j_i$
we find that $\rmPsi_f(X,j_i)\bmod p$ does not have exactly $n$ roots
$f_i^{(1)},\ldots,f_i^{(n)}$, we exclude the prime $p$ from our
computations.  The number of such $p$ is finite and may be bounded in
terms of the discriminants of the polynomials $\rmPsi_f(X,\alpha)$ as
$\alpha$ ranges over the roots of $H_D[f]$.  We then compute the polynomial
$H(X)=\prod_{i=1}^{h} \prod_{r=1}^n \bigl( X - f_i^{(r)} \bigr)$ of degree $nh$
in $\Fp[X]$.  After doing this for sufficiently many primes $p$, we
can lift the coefficients by Chinese remaindering to the integers.
The resulting $H$ is a product of $n$ distinct class polynomials, all of which
may be obtained by factoring $H$ in $\Z[X]$.  Under suitable heuristic
assumptions (including the GRH), the total time to compute $H_D[f]$ is
quasi-linear in $|D|$, including the time to factor $H$.

This approach is practically efficient only when $n$ is small,
but then it can be quite useful.  A notable example is the modular function $g$ for which
\[
\rmPsi_g(X,J)=(X^{12}-6X^6-27)^3-JX^{18}.
\]
This function was originally proposed by Atkin, and is closely related to certain
class invariants of Ramanujan \cite[Thm.~4.1]{BerndtChan:Ramanujan}.
The function $g$ yields class invariants when $D\equiv13\bmod 24$.
In terms of our generic algorithm,
we have $n=2$, and for $p\equiv 2\bmod 3$ we get
exactly two roots of $\rmPsi_g(X,j_i)\bmod p$, which differ only in sign.
Thus $H(X)=H_D[g^2](X^2)=H_D[g](X)H_D[g](-X)$, and from this we easily obtain
$H_D[g^2]$, and also $H_D[g]$ if desired.

\section{Computational Results}
\label {sec:implementation}
This section provides performance data for the
techniques developed above.  We used AMD Phenom II 945 CPUs
clocked at 3.0~GHz
for our tests; the software was implemented
using the \texttt {gmp} \cite{gmp} and \texttt {zn\_poly} \cite{Harvey:zn_poly} libraries,
and compiled with \texttt {gcc} \cite{gcc}.

To compute the class polynomial $H_D[f]$, we require
a bound on the size of its coefficients.  Unfortunately, provably accurate bounds for
functions~$f$ other than~$j$ are generally unavailable.  As a heuristic, we take
the bound $B$ on the coefficients of $H_D$ given by \cite[Lem.~8]{Sutherland:HilbertClassPolynomials},
divide $\log_2 B$ by the asymptotic height factor $c(f)$, and add a ``safety margin" of 256 bits.
We note that with the CM method, the correctness of the final result can be efficiently
and unconditionally confirmed~\cite{BissonSutherland:Endomorphism}, so we are generally
happy to work with a heuristic bound.

\subsection{Class polynomial computations using the CRT method}

Our first set of tests measures the improvement relative to previous
computations with the CRT method.  We used discriminants related to the
construction of a large set of pairing-friendly elliptic curves, see \cite[\S 8]{Sutherland:HilbertClassPolynomials} for details.
We reconstructed many of these curves, first using the Hilbert class polynomial $H_D$,
and then using an alternative class polynomial $H_D[f]$.  In each case we used
the explicit CRT to compute $H_D$ or $H_D[f]$ modulo a large prime $q$ (170 to 256 bits).

Table~\ref {tab1} gives results for four discriminants with $|D|\approx 10^{10}$,
three of which appear in \cite[Table~2]{Sutherland:HilbertClassPolynomials}.  Each column
lists times for three class polynomial computations.  First, we give the total time
$T_{\rm tot}$ to compute $H_D\bmod q$, including
the time $T_{\rm enum}$ spent enumerating $\Ell_D(\Fp)$, for all the small primes $p$,
using Algorithm~\ref{alg:crtjenum} as it appears in \S\ref{ssec:crtjenum}.  We then list the
times $T'_{\rm enum}$ and $T'_{\rm tot}$ obtained when Algorithm~\ref{alg:crtjenum} is modified to
use gcd computations whenever it is advantageous to do so, as explained in \S\ref{ssec:gcd}.
The gcd approach typically speeds up Algorithm~\ref{alg:crtjenum} by a factor
of~$2$ or more.

For the third computation we selected a function $f$ that yields class invariants
for $D$, and computed $H_D[f]\bmod q$.  This polynomial can be used in place of $H_D$
in the CM method (one extracts a root $x_0$ of $H_D[f]\bmod q$, and then extracts
a root of $\rmPsi_f(x_0,J)\bmod q$).
For each function $f$ we give a ``size factor", which approximates the ratio
of the total size of $H_D$ to $H_D[f]$ (over $\Z$).  In the first three examples this is just
the height factor $c(f)$, but in Example 4 it is $4c(f)$ because the prime 59 is ramified and
we actually work with the square root of $H_D[A_{59}]$, as noted in \S\ref{ssec:direct},
reducing both the height and degree by a factor of~2.

We then list the speedup $T'_{\rm tot}/T'_{\rm tot}[f]$ attributable to computing $H_D[f]$ rather
than $H_D$.  Remarkably, in each case this speedup is about twice what one would
expect from the height factor.  This is explained by a particular feature of the CRT method:
The cost of computing $H_D\bmod p$ for small primes $p$ varies significantly, and, as explained
in \cite[\S 3]{Sutherland:HilbertClassPolynomials}, one can accelerate the CRT method with a
careful choice of primes.  When fewer
small primes are needed, we choose those for which Step 1 of Algorithm~\ref{alg:crtj}
can be performed most quickly.

The last line in Table~\ref {tab1} lists the total speedup
$T_{\rm tot}/T'_{\rm tot}[f]$ achieved.

\begin{table}
\begin{center}
\begin{tabular}{@{}l@{}rrrr@{}}
&Example 1& Example 2& Example 3&Example 4\\
\midrule
$|D|$ & $13569850003$ & $\quad\medspace 11039933587$ & $\quad\medspace 12901800539$ & $\quad\quad \textsc{12042704347}$\\
$h(D)$& 20203 & 11280 & 54706 & 9788\\\vspace{2pt}
$\left\lceil\log_2 B\right\rceil$& 2272564& 1359134 & 5469776 & 1207412\\
$(\ell_1^{r_1},\ldots,\ell_k^{r_k})$&$(7^{20203})$&$(17^{1128},19^{10})$&$(3^{27038},5^2)$&$(29^{2447},31^2,43^2)$\\
\midrule
$T_{\rm enum}$ (roots)                   &  6440 & 10200 & 10800 & 21700\\
$T_{\rm tot}$                            & 19900 & 23700 & 52200 & 42400\\
\midrule
$T'_{\rm enum}$ (gcds)                   &  2510 &  2140 &  3440 &  4780\\
$T'_{\rm tot}$                           & 15900 & 15500 & 44700 & 25300\\
\midrule
Function $f$                         & $A_{71}$& $A_{47}$& $A_{71}$& $A_{59}$\\
Size factor                            &    36 &    24 &    36 &   120*\\
\vspace{6pt}
$T'_{\rm tot}[f]$                          &   213 &   305 &   629 & 191\\

Speedup ($T'_{\rm tot}/T'_{\rm tot}[f]$)   &    75 &    51 &    71 &  132\\
Speedup ($T_{\rm tot}/T'_{\rm tot}[f]$)    &    {\bf 93} &    {\bf 78} &    {\bf 83} &  {\bf 222}\\
\bottomrule
\end{tabular}
\\
\vspace{6pt}
\caption{Example class polynomial computations (times in CPU seconds)}
\label {tab1}
\end{center}
\end{table}

\subsection{Comparison to the complex analytic method}
\label{ssec:compare}

Our second set of tests compares the CRT approach to the complex
analytic method.  For each of the five discriminants listed in Table~\ref {tab2}
we computed class polynomials $H_D[f]$ for the double $\eta$-quotient
$\w_{3,13}$ and the Weber $\f$ function, using both the CRT approach
described here, and the implementation \cite{cm} of the complex analytic method
as described in \cite{Enge:FloatingPoint}.
With the CRT we computed $H_D[f]$ both over $\Z$ and modulo a 256-bit prime~$q$;
for the complex analytic method these times are essentially the same.

\begin{table}
\begin{center}
\begin{tabular}{@{}rrcrrcrrcrr@{}}
&&&\multicolumn{2}{c}{complex analytic}&&\multicolumn{2}{c}{CRT}&&\multicolumn{2}{c}{CRT mod $q$}\\
\cmidrule(r){4-5}\cmidrule(r){7-8}\cmidrule(r){10-11}
$|D|$&$\qquad h(D)$&$\qquad$&$\w_{3,13}$&$\qquad\quad\f\medspace$&$\qquad\medspace$&$\w_{3,13}$&$\qquad\medspace\f\medspace$&$\qquad\medspace$&$\w_{3,13}$&$\qquad\f\medspace$\\
\midrule
   6961631 &   5000 &&    15 &   5.4 &&  2.2 &  1.0 &&  2.1 &  1.0\\
  23512271 &  10000 &&   106 &    33 &&   10 &  4.1 &&  9.8 &  4.0\\
  98016239 &  20000 &&   819 &   262 &&   52 &   22 &&   47 &   22\\
 357116231 &  40000 &&  6210 &  1900 &&  248 &  101 &&  213 &   94\\
2093236031 & 100000 && 91000 & 27900 && 2200 &  870 && 1800 &  770\\
\bottomrule
\end{tabular}
\\
\vspace{6pt}
\caption{CRT vs. complex analytic (times in CPU seconds)}
\label {tab2}
\end{center}
\end{table}

We also tested a ``worst case" scenario for the CRT approach: the discriminant $D=-85702502803$,
for which the smallest non-inert prime is $\ell_1=109$.  Choosing
the function most suitable to each method, the complex analytic method computes
$H_D[\w_{109,127}]$ in 8310 seconds, while the CRT method computes $H_D[A_{131}]$ in
7150 seconds.  The CRT approach benefits from the attractive height
factor of the Atkin functions, $c(A_{131})=33$ versus $c(\w_{109,127})\approx 12.4$,
and the use of gcds in Algorithm~\ref {alg:crtjenum}.
Without these improvements,
the time to compute $H_D$ with the CRT method is 1460000 seconds. The
techniques presented here yield more than a 200-fold speedup in this example.

\subsection{A record-breaking CM construction}

To test the scalability of the CRT approach, we constructed an elliptic curve
using $|D|=1000000013079299 > 10^{15}$,
with $h(D)=10034174>10^7$.
This yielded a curve $y^2=x^3-3x+c$ of prime order $n$ over the prime field $\Fq$, where
\scriptsize
\begin{align*}
c &= 12229445650235697471539531853482081746072487194452039355467804333684298579047;\\
q &= 28948022309329048855892746252171981646113288548904805961094058424256743169033;\\
n &= 28948022309329048855892746252171981646453570915825744424557433031688511408013.
\end{align*}
\normalsize
This curve was obtained by computing the square root of $H_D[A_{71}]$ modulo $q$,
a polynomial of degree $h(D)/2=5017087$.  The height bound of 21533832 bits
was achieved with 438709 small primes $p$, the largest of which was
53 bits in size.  The class polynomial computation took slightly less than a week
using 32 cores, approximately 200 days of CPU time.
Extracting a root over $\Fq$ took 25 hours of CPU time using NTL \cite{Shoup:NTL}.

We estimate that the size of $\sqrt{H_D[A_{71}]}$ is over 13 terabytes,
and that the size of the Hilbert class polynomial $H_D$ is nearly 2 petabytes.
The size of $\sqrt{H_D[A_{71}]}\bmod q$, however, is under 200 megabytes, and less
than 800 megabytes of memory (per core) were needed to compute it.

\bibliographystyle{plain}

\end{document}